\documentclass[11pt,letterpaper]{amsart}

\usepackage{geometry}
\usepackage[english]{babel}
\usepackage{amsmath,amssymb,mathtools,enumerate,latexsym,graphicx,cite}
\usepackage{xcolor}
\usepackage{hyperref}
\hypersetup{
  colorlinks=true,
  linkcolor=black,
  citecolor=black,
  urlcolor=black,
  pdftitle={Codimension-one holomorphic Anosov diffeomorphisms},
  pdfauthor={Jiesong Zhang}
}

\newtheorem{theorem}{Theorem}[section]
\newtheorem{lemma}[theorem]{Lemma}

\newtheorem{proposition}[theorem]{Proposition}

\newtheorem{definition}[theorem]{Definition}
\theoremstyle{remark}

\newtheorem{mainthm}{\bfseries Theorem}

\newcommand{\C}{\mathbb C}
\newcommand{\R}{\mathbb R}
\newcommand{\Z}{\mathbb Z}

\newcommand{\id}{\operatorname{id}}

\newcommand{\ol}{\overline}
\newcommand{\wt}{\widetilde}

\newif\ifshowchanges
\showchangestrue
\newcommand{\revisioncolor}{%
  \ifshowchanges
    \hypersetup{linkcolor=red,citecolor=red,urlcolor=red}\color{red}%
  \fi}

\title[Codimension-one holomorphic Anosov diffeomorphisms]
{Codimension-one holomorphic Anosov diffeomorphisms}

\author{Jiesong Zhang}
\address{Department of Mathematics\\
Kungliga Tekniska h\"ogskolan, Lindstedtsv\"agen 25\\
SE-100 44 Stockholm\\
Sweden}
\email{jiesongz@kth.se}

\begin{document}

\begin{abstract}
We prove that every codimension-one holomorphic Anosov diffeomorphism of a compact connected complex manifold is biholomorphically conjugate to a hyperbolic automorphism of a complex torus. This verifies a conjecture of Ghys \cite{ghys_1995} in the codimension-one case and, in particular, in complex dimension three.
\end{abstract}

\maketitle

\tableofcontents

\section{Introduction}
A diffeomorphism $f\colon M\to M$ of a connected compact Riemannian manifold is called \emph{Anosov} if there is a continuous $Df$-invariant splitting
\[
  TM=E^s\oplus E^u
\]
and an integer $k\geq1$ such that, for every $x\in M$,
\[
  \bigl\|Df^k|_{E^s(x)}\bigr\|<1
  \quad\text{and}\quad
  \bigl\|(Df^k|_{E^u(x)})^{-1}\bigr\|<1.
\]

Franks and Manning proved that an Anosov diffeomorphism on an infra-nilmanifold is
topologically conjugate to a hyperbolic affine automorphism \cite{franks,manning_1974}. The conjugacy and its inverse are H\"older continuous, but the conjugacy need not be a $C^1$ diffeomorphism in general:
a $C^1$ conjugacy identifies the eigenvalues of the corresponding derivative return maps at periodic points,
whereas these eigenvalues can be changed by arbitrarily small local $C^\infty$ perturbations. A conjecture going back to Anosov and Smale asserts that every compact manifold supporting an Anosov diffeomorphism is homeomorphic to an infra-nilmanifold \cite{smale_1967}.

The dynamics of holomorphic diffeomorphisms on complex manifolds are often more rigid than those of general differentiable maps, owing to the compatibility between the dynamics and the complex structure. In \cite{ghys_1995}, Ghys proposed the following conjecture.\smallskip

\begin{quote}\itshape Every holomorphic Anosov diffeomorphism is biholomorphically conjugate to an affine automorphism of a complex infra-nilmanifold.\end{quote}

Ghys verified the conjecture for complex surfaces \cite[Theorem~A]{ghys_1995}. Cantat \cite[Theorem~1.4]{cantat_2004} proved the corresponding classification for projective manifolds, assuming either that both invariant distributions are holomorphic or that the diffeomorphism has complex codimension one. Zhang \cite{zhang_2026} proved a regularity bootstrap and algebraic classification for holomorphic Anosov diffeomorphisms that are bi-Lipschitz conjugate to affine models.

In higher dimensions, the conjecture remains open in general. In this note, we prove it for holomorphic Anosov diffeomorphisms of complex codimension one, meaning that $\dim_\mathbb C E^s=1$ or $\dim_\mathbb C E^u=1$.
\begin{mainthm}\label{theorem rigidity}
Let \(f\colon X\to X\) be a codimension-one holomorphic Anosov diffeomorphism of a compact complex manifold. Then $X$ is biholomorphic to a complex torus, and $f$ is biholomorphically conjugate to a hyperbolic holomorphic torus automorphism.
\end{mainthm}
In particular, Theorem~\ref{theorem rigidity} verifies Ghys's conjecture for compact complex threefolds: since both invariant distributions are nonzero, one of them must have complex dimension one.

\subsection{Strategy of the proof}
In this subsection, let $f\colon X\to X$ be a codimension-one holomorphic Anosov diffeomorphism of a compact complex manifold of complex dimension $d$. Replacing $f$ by $f^{-1}$ if necessary, we assume that $\dim_\mathbb C E^s=1$.
\subsubsection*{Topological conjugacy} The first step is to establish the topological conjugacy.
\begin{theorem}\label{theorem topo}
Let $f\colon X\to X$ be a holomorphic Anosov diffeomorphism of a compact complex manifold of complex dimension $d\ge2$, with $\dim_\mathbb C E^s=1$. Then $X$ is homeomorphic to $\mathbb T^{2d}$, and $f$ is topologically conjugate to a hyperbolic toral automorphism.
\end{theorem}
The case $d=2$ is covered by \cite[Theorem A]{ghys_1995}. For $d \geq 3$, Ghys \cite[Theorem~B]{ghys_1995} proved this statement under the additional assumption of topological transitivity. The transitivity assumption is only used to show that the unstable holonomy maps are globally defined on every stable leaf. In this paper, we establish the global existence of the unstable holonomy by using a specific geometrically defined non-stationary linearization of the stable leaves, with respect to which the unstable holonomy maps are linear and thus globally defined. This enables us to remove the topological transitivity assumption from Ghys's Theorem B. 

We note that Theorem \ref{theorem topo} does not follow from the Franks--Newhouse classification of codimension-one Anosov diffeomorphisms \cite{franks,newhouse}, since in our setting $E^s$ and $E^u$ have real dimension at least two.

\subsubsection*{Holomorphic conjugacy}

The second step is to improve the regularity of the conjugacy. A difficulty is that a complex manifold homeomorphic to $\mathbb T^{2d}$ need not be a complex torus when $d\ge3$: real tori admit nonstandard complex structures; see \cite{catanese_2002,catanese_oguiso_peternell_2010}.\footnote{In contrast, a compact complex surface homeomorphic to $\mathbb T^4$ is a complex torus, by the classification of compact complex surfaces.} This leads to the following question.
\begin{quote}\itshape
Can a nonstandard complex structure on $\mathbb T^{2d}$, $d\ge3$, admit a holomorphic Anosov diffeomorphism?
\end{quote}
We expect the answer to be negative. Theorem~\ref{theorem rigidity} proves this in complex dimension three and, more generally, in the codimension-one case. By contrast, Xu and Zhang \cite{xu_zhang_2025} constructed holomorphic partially hyperbolic systems on real tori with nonstandard complex structures in complex dimensions at least five.

The following theorem gives a negative answer to the question above under the additional assumption that $E^s$ and $E^u$ are holomorphic. 

\begin{theorem}\label{theorem anosov torus}
Let $f\colon X\to X$ be a holomorphic Anosov diffeomorphism of a compact complex manifold. Assume that $X$ is homeomorphic to a torus and that the invariant distributions $E^s$ and $E^u$ are holomorphic. Then $X$ is biholomorphic to a complex torus, and $f$ is biholomorphically conjugate to a
holomorphic torus automorphism.
\end{theorem}
The main idea is to decompose deck transformations of the universal cover into their stable and unstable components and use the holomorphicity of both foliations to show that these components are biholomorphic. Their closure gives a transitive action by holomorphic translations, forcing the complex structure to be translation-invariant.

To apply Theorem \ref{theorem anosov torus}, we establish the holomorphicity of the invariant distributions. 

\begin{theorem}\label{theorem distribution}
Let $f\colon X\to X$ be a holomorphic Anosov diffeomorphism of a compact complex manifold with $\dim_\mathbb C E^s=1$. Then both $E^s$ and $E^u$ are holomorphic subbundles of $T^{1,0}X$.
\end{theorem}
The invariant distributions $E^s$ and $E^u$ are a priori H\"older continuous \cite{hirsch_pugh_shub_1977}. In the setting of Theorem \ref{theorem distribution}, Ghys \cite[Proposition~2.2]{ghys_1995} proved that $E^u$ is holomorphic by estimating the quasiconformal distortion of the unstable holonomy maps. This method does not extend directly to $E^s$ since holomorphicity and conformality are no longer equivalent in higher complex dimensions. As an alternative, we prove the holomorphicity of $E^s$ directly by an $L^2$-estimate for the antiholomorphic derivatives of suitable smooth approximations. 

Theorem~\ref{theorem rigidity} then follows from Theorems~\ref{theorem topo}, \ref{theorem anosov torus}, and~\ref{theorem distribution}. 

\subsection{Obstructions in higher codimension}
One may ask whether the strategy of the proof of Theorem~\ref{theorem rigidity} extends to higher codimension. The three steps of the proof encounter different obstructions.

The topological conclusion of Theorem~\ref{theorem topo} does not extend to higher codimension in general. Cantat constructed a holomorphic Anosov diffeomorphism on a compact complex fourfold that is not homeomorphic to a torus, with $\dim_\mathbb C E^s=\dim_\mathbb C E^u=2$; see \cite[Example~1.5]{cantat_2004}. Section~\ref{sec example} gives another example on a non-torus nilmanifold. Thus, a general classification must allow algebraic models beyond tori. 

Theorem~\ref{theorem anosov torus} itself does not require the codimension-one assumption. Its proof, however, uses the additive structure of the toral model and the density of the stable and unstable projections of the lattice to construct a transitive group of holomorphic translations. This argument does not extend directly to general infra-nilmanifolds. It is natural to ask whether an analogous classification by affine automorphisms holds for manifolds homeomorphic to infra-nilmanifolds, under the assumption that both invariant distributions are holomorphic. However, even if such a generalization holds, the holomorphicity assumption remains an obstruction to applying it in full generality.

Indeed, Section~\ref{sec example} constructs a holomorphic Anosov automorphism of a compact nilmanifold with a left-invariant complex structure whose stable and unstable distributions are both smooth but non-holomorphic. Thus, the regularity conclusion of Theorem~\ref{theorem distribution} fails in higher codimension. This example is already an algebraic model consistent with Ghys's conjecture, showing that the conjectured classification must accommodate non-holomorphic invariant distributions.

\subsection{Structure of the paper}
Section~\ref{tt:section} proves Theorem~\ref{theorem topo}, Section~\ref{sec torus Anosov} proves Theorem~\ref{theorem anosov torus}, and Section~\ref{sec:distribution} proves Theorem~\ref{theorem distribution}. Section~\ref{sec example} constructs a holomorphic Anosov automorphism whose invariant distributions are both non-holomorphic. 

\subsection*{Declaration of AI Use}
The author used an internal version of Intern-S2 to assist in developing the mathematical ideas and arguments presented in this paper through multiple rounds of dialogue. The author is grateful to Shanghai AI Lab, especially Bowen Zhou and the Intern team for their support. The author has independently verified all mathematical arguments and takes full responsibility for the manuscript.

\subsection*{Formalised proof}
The central results of this paper, Theorem~A and Theorems~1.1, 1.2, and~1.3, together with Ghys's results \cite{ghys_1995} used in this paper, have been formalised in Lean by the Bin Dong--Ruochuan Liu team in Peking University, using the open-source formalisation agent SHEAF (\url{https://github.com/frenzymath/SHEAF}) developed by the same team. The formalisation, including the human verification of the final theorem, was carried out by Jiedong Jiang and Bin Wu. The source code is available at \url{https://github.com/frenzymath/HolomorphicAnosov}. The author is grateful to the team for their support.

\subsection*{Acknowledgements}

The question addressed in this paper was already under consideration while the author and Disheng Xu were completing their joint work \cite{xu_zhang_2025}. The author thanks Disheng Xu for continued discussions on this question and for facilitating his stay as a visiting scholar at Shanghai AI Lab, when the main ideas of this work took shape. The author thanks Junyi Xie for his comments on the problem and on the first version of this paper, and Danijela Damjanović for her encouragement and helpful suggestions. The author also thanks Aaron Brown, Zhuchao Ji, Jiedong Jiang, Keiji Oguiso, Zexuan Ouyang, Jiyao Tang, Haiteng Zhao, and Ziyi Zhao for helpful discussions related to this problem. This work was supported in part by 2024YFA1015100, NSFC 12090010 and 12090015. 

\section{Proof of Theorem~\ref{theorem topo}}
\label{tt:section}
Write $d=\dim_\mathbb C X$. The case $d=2$ follows from \cite[Theorem~A]{ghys_1995}, so throughout this section we assume $d\ge3$. By \cite[Proposition~2.2]{ghys_1995}, applied to $f^{-1}$, the unstable bundle $E^u$ is holomorphic. Thus
\[
    L:=T^{1,0}X/E^u
\]
is a holomorphic line bundle. Denote the quotient map by $\pi\colon T^{1,0}X\to L$. We can also regard $\pi$ as an $L$-valued $(1,0)$-form. Set
\[
    S_n(x)=D_xf^n|_{E^s_x},\qquad
    U_n(x)=D_xf^n|_{E^u_x},\qquad
    Q_n(x)\colon L_x\longrightarrow L_{f^nx},
\]
where $Q_n$ is the quotient map induced by $Df^n$ with
\[
    \pi_{f^nx}\circ D_xf^n=Q_n(x)\circ\pi_x.
\]
Fix a smooth Hermitian metric on $T^{1,0}X$ and equip $E^u$ and $L$ with the induced and quotient metrics. Uniform hyperbolicity gives constants $C>0$, $0<\tau<1$, and $\lambda>1$ such that, for every $x\in X$ and $n\ge1$,
\begin{equation}\label{tt:bounds}
    \begin{gathered}
    \|S_n(x)\|\le C\tau^n,\qquad
    \|Q_n(x)\|\le C\tau^n,\\
    m(U_n(x))\ge C^{-1}\lambda^n,\qquad
    \|D_xf^n\|\le C\|U_n(x)\|.
    \end{gathered}
\end{equation}
Here $m(B)=\|B^{-1}\|^{-1}$. Constants denoted by $C$ may change from one occurrence to the next.

\subsection{A compatible invariant connection}
The involutivity of $E^u$ gives a canonical way to differentiate sections of $L=T^{1,0}X/E^u$ along $E^u$. For a smooth $(1,0)$ vector field $Y$ tangent to $E^u$ and a smooth section $s$ of $L$, the \emph{Bott partial connection} is given by 
\begin{equation}\label{nt:bott}
    \nabla^{\mathrm B}_Y s:=\pi[Y,\widetilde s],
    \qquad \pi(\widetilde s)=s,
\end{equation}
where $\widetilde s$ is any smooth local lift of $s$ to $T^{1,0}X$. The expression is independent of the lift $\widetilde s$: two lifts differ by a section $Z$ of $E^u$, and $\pi[Y,Z]=0$ because $E^u$ is involutive. 
\begin{definition}[Compatible connection]\label{nt:compatible}
A complex connection $\nabla$ on $L$ is called \emph{compatible} if its $(0,1)$-part is $\bar\partial_L$ and its restriction to $E^u$ is $\nabla^{\mathrm B}$.
\end{definition}
In local holomorphic foliation coordinates $(z,w_1,\ldots,w_{d-1})$ with $E^u=\ker dz$, set $e=\pi(\partial_z)$. A connection is determined in this frame by a complex-valued one-form $A$ through
\begin{equation}\label{eq connect}
      \nabla e=A\otimes e,\qquad
    \nabla(ue)=(du+uA)\otimes e.
\end{equation}
Compatibility is equivalent to
\begin{equation}\label{eq connect2}
    A=a\,dz
\end{equation}
for a scalar function $a$. Indeed, $\nabla^{0,1}=\bar\partial_L$ excludes the $(0,1)$-part of $A$ as $e$ is holomorphic, and $\nabla^{\mathrm B}_{\partial_{w_j}}e=\pi[\partial_{w_j},\partial_z]=0$ excludes the $dw_j$-components. The local connections with $A=0$ are therefore compatible. A smooth partition of unity combines them into a global smooth compatible connection $\nabla_0$. We will use this local expression several times during the proof.

The bundle map $Q_1\colon L\to L$ covering $f$ identifies a pulled-back connection with a connection on $L$. More explicitly, put
\[
 (f_\#s)(y)=Q_1(f^{-1}y)s(f^{-1}y).
\]
For a smooth vector field $V$ and section $s$, our convention is
\begin{equation}\label{nt:pullback_connection}
 ((f^*\nabla)_V s)(x)
 =Q_1(x)^{-1}\bigl(\nabla_{f_*V}(f_\#s)\bigr)(fx).
\end{equation}
In particular, the difference of two connections pulls back as an ordinary scalar one-form, since conjugation acts trivially on $\operatorname{End}(L)$.

\begin{lemma}\label{lemma continuous connection}
The line bundle $L$ admits a unique continuous compatible $f$-invariant connection $\nabla$.
\end{lemma}
\begin{proof}
Define
\[
    \nabla_n=(f^n)^*\nabla_0,\qquad
    \alpha=f^*\nabla_0-\nabla_0.
\]
The Leibniz rule shows that $\alpha$ is $C^\infty$-linear in both the vector field and the section. Thus $\alpha$ is an $\operatorname{End}(L)$-valued one-form. Since $L$ is a complex line bundle, we have the canonical identifications 
\[
    \operatorname{End}(L)\cong L\otimes L^*\cong X\times\mathbb C.
\]
We therefore regard $\alpha$ as an ordinary complex-valued one-form.

Since $f$ is holomorphic and preserves $E^u$, pullback preserves both $\bar\partial_L$ and the Bott partial connection. Consequently, each $\nabla_n$ is compatible, and
\[
    \alpha\in\Omega^{1,0}(X),\qquad \alpha|_{E^u}=0.
\]
The form $\alpha$ is smooth but need not be holomorphic. 

For $v\in T^{1,0}_xX$, write $v=v_s+v_u$ according to the invariant splitting. Then
\[
    \bigl((f^n)^*\alpha\bigr)_x(v)
    =\alpha_{f^nx}(S_n(x)v_s).
\]
The projections onto $E^s$ and $E^u$ are uniformly bounded, so \eqref{tt:bounds} implies
\begin{equation}\label{tt:connection_convergence}
    \|\nabla_{n+1}-\nabla_n\|_{C^0}
    =\|(f^n)^*\alpha\|_{C^0}
    \le C\tau^n\|\alpha\|_{C^0}.
\end{equation}
The summability of the right-hand side shows that $\nabla_n$ converges uniformly to a continuous compatible connection $\nabla$. Passing to the limit in $f^*\nabla_n=\nabla_{n+1}$ gives $f^*\nabla=\nabla$.

If $\nabla'$ is another continuous compatible invariant connection, then $\beta=\nabla'-\nabla$ is a scalar $(1,0)$-form vanishing on $E^u$. The same calculation gives
\[
    \beta=(f^n)^*\beta,\qquad
    \|\beta\|_{C^0}\le C\tau^n\|\beta\|_{C^0}.
\]
Letting $n\to\infty$ yields $\beta=0$.
\end{proof}

\begin{lemma}\label{lemma connection holomorphic}
The connection $\nabla$ is holomorphic and flat.
\end{lemma}
\begin{proof}
As shown in \eqref{eq connect} and \eqref{eq connect2}, in local holomorphic foliation coordinates $(z,w_1,\ldots,w_{d-1})$ with $E^u=\ker dz$ and $e=\pi(\partial_z)$, the compatible connections $\nabla_n$ and $\nabla$ have connection one-forms of the following type: 
\begin{equation}\label{eq coordinate nabla}
    \nabla_ne=A_n\otimes e,\qquad A_n=a_n\,dz,
    \qquad \nabla e=A\otimes e,\quad A=a\,dz.
\end{equation}
Therefore, the curvature of the connection $\nabla_n$ is  given by
\[R_n=d A_n+A_n\wedge A_n=dA_n=d a_n\wedge dz.\]
Naturality of curvature gives $R_n=(f^n)^*R_0$; the usual conjugation by $Q_n$ disappears because $L$ has rank one. In a foliation chart around $f^nx$, write $R_0=da_0\wedge dz$. On its inverse image,
\[
    R_n=(f^n)^*(da_0)\wedge(f^n)^*(dz).
\]
Since $dz$ annihilates $E^u$, it factors through the quotient $L$, and therefore
\[
    |(f^n)^*dz|_x\le C\|Q_n(x)\|,
    \qquad |(f^n)^*da_0|_x\le C\|D_xf^n\|.
\]

Fix a finite cover $X=\bigcup_{i=1}^N V_i$, where each $V_i$ is relatively compact in a holomorphic foliation chart $U_i$.
Writing $R_0=da_{0,i}\wedge dz_i$ on $U_i$, the one-forms $da_{0,i}$ and the quotient covectors induced by $dz_i$ are uniformly bounded on the sets $V_i$.
For each $x$ and $n$, choose $i$ such that $f^nx\in V_i$ and apply the preceding estimates in this chart.
Thus we can choose a uniform constant $C$, independent of both $x$ and $n$, such that
\begin{equation}\label{nt:curvature_pointwise}
    |R_n(x)|^2
    \le C\|Q_n(x)\|^2\|D_xf^n\|^2
    \le C\|Q_n(x)\|^2\|U_n(x)\|^2.
\end{equation}

Let $\nu$ be the real volume form of the chosen Hermitian metric. In unitary frames adapted to $E^u\oplus(E^u)^\perp$, the derivative has the block form
\[
    D_xf^n=
    \begin{pmatrix}U_n(x)&b_n(x)\\0&q_n(x)\end{pmatrix},
\]
where $|q_n|=\|Q_n\|$. For any complex-linear map $B$, its real determinant is $\det_{\mathbb R}B=|\det_{\mathbb C}B|^2$.  Thus
\begin{equation}\label{nt:quotient_jacobian}
    J_\nu(f^n)(x)
    =\|Q_n(x)\|^2|\det_{\mathbb C}U_n(x)|^2.
\end{equation}
If $\sigma_1\ge\cdots\ge\sigma_{d-1}>0$ are the singular values of $U_n(x)$, then
\[
    |\det_{\mathbb C}U_n(x)|
    =\prod_{j=1}^{d-1}\sigma_j
    \ge\|U_n(x)\|\,m(U_n(x))^{d-2}.
\]
Combining this with \eqref{tt:bounds}, \eqref{nt:curvature_pointwise}, and \eqref{nt:quotient_jacobian} gives
\begin{equation}\label{nt:curvature_jacobian}
    |R_n(x)|^2
    \le C\,m(U_n(x))^{-2(d-2)}J_\nu(f^n)(x)
    \le C\lambda^{-2(d-2)n}J_\nu(f^n)(x).
\end{equation}
Integrating this inequality gives
\begin{equation}\label{nt:curvature_l2}
    \begin{aligned}
    \|R_n\|_{L^2(\nu)}^2
    &\le C\lambda^{-2(d-2)n}\int_XJ_\nu(f^n)\,d\nu\\
    &=C\lambda^{-2(d-2)n}\nu(X)\longrightarrow0.
    \end{aligned}
\end{equation}
Here $d\ge3$ is essential. The volume form $\nu$ need not be $f$-invariant.

On each fixed foliation chart, $A_n\to A$ locally uniformly and $dA_n=R_n\to0$ in $L^2$. Both convergences hold in the sense of distributions, and continuity of distributional differentiation gives $dA=0$.

Since $A=a\,dz$ has type $(1,0)$, the components $\partial A$ and $\bar\partial A$ have distinct types and vanish separately. Thus $(\bar\partial a)\wedge dz=0$, which implies $\bar\partial a=0$ coefficientwise. The distributional Cauchy--Riemann theorem \cite[Chapter~I, Lemma~(3.29)(a)]{demailly_2012} shows that $a$ is holomorphic. Hence $\nabla$ is holomorphic and its curvature $dA$ vanishes.
\end{proof}

\subsection{Affine unstable holonomy}
Let $\nabla$ be the flat holomorphic connection constructed above. We will use this connection to define a normal form coordinate along $W^s$ for which the unstable holonomy is linear. 
\begin{lemma}\label{lemma transverse coordinates}
For every $x \in X$, there is a neighborhood $U$, a nowhere-vanishing $\nabla$-parallel holomorphic section $\sigma$ of $L|_U$, and a holomorphic submersion $t\colon U\to\mathbb C$ such that
\[
 \pi=dt\otimes\sigma,\qquad \ker dt=E^u.
\]
\end{lemma}

\begin{proof}
Choose a sufficiently small simply connected holomorphic foliation box $U$ with coordinates $(z,w_1,\ldots,w_{d-1})$ such that $E^u=\ker dz$. Put $e=\pi(\partial_z)$, so that $\pi=dz\otimes e$.

Recall that for an $L$-valued one-form, the exterior covariant derivative is given by
\begin{equation}\label{eq general}
    d^\nabla(\omega\otimes e)
    =d\omega\otimes e-\omega\wedge\nabla e.
\end{equation}
Since $\nabla e=a\,dz\otimes e$, we have
\begin{equation}\label{tt:covariantly_closed}
    d^\nabla\pi
    =\bigl(d(dz)-dz\wedge(a\,dz)\bigr)\otimes e=0.
\end{equation}

Flatness and simple connectedness give a nowhere-vanishing parallel section $\sigma$ of $L|_U$. It is holomorphic because $\nabla^{0,1}=\bar\partial_L$. Write $\pi=\omega\otimes\sigma$, where $\omega$ is a holomorphic one-form. Since $\nabla\sigma=0$, \eqref{eq general} and
\eqref{tt:covariantly_closed} give
\[
    0=d^\nabla\pi
    =d^\nabla(\omega\otimes\sigma)
    =d\omega\otimes\sigma.
\]
As $\sigma$ is nowhere vanishing, we obtain $d\omega=0$. After shrinking $U$, the holomorphic Poincar\'e lemma yields
a holomorphic function $t\colon U\to\mathbb C$ such that $dt=\omega$. Thus $\pi=dt\otimes\sigma$.
Since $\pi$ is fiberwise surjective, $dt$ is nowhere vanishing.
Hence $t$ is a holomorphic submersion and
$\ker dt=\ker\pi=E^u|_U$.
\end{proof}
For $x\in U$, set
\[
 B_x=dt_x|_{E_x^s}\colon E_x^s\longrightarrow\mathbb C.
\]
This map is an isomorphism. On a sufficiently small stable plaque through $x$, define
\begin{equation}\label{tt:identification}
 k_x(p)=B_x^{-1}\bigl(t(p)-t(x)\bigr).
\end{equation}
We then prove this definition is independent of the choice of $(t,\sigma)$. 
\begin{lemma}\label{lemma:uniform_local_coordinates}
The maps $k_x$ are independent of the chosen pairs $(t,\sigma)$ and depend continuously on $x$. Moreover, there exist constants $r,\varepsilon>0$ such that for every $x\in X$,
\[
    k_x\colon W^s_r(x)\longrightarrow k_x(W^s_r(x))\subset E_x^s
\]
is a biholomorphism satisfying
\[
    k_x(x)=0,\qquad D_xk_x=\id. 
\]
Here $W^s_r(x)$ denotes the ball of radius $r$ in the stable leaf. Moreover, the image of $k_x$ contains the ball $B(0,\varepsilon)\subset E_x^s$ for every $x\in X$.
\end{lemma}

\begin{proof}
Choose a finite cover $X=\bigcup_{i=1}^N U_i$, where each $U_i$ carries a pair
$(t_i,\sigma_i)$ as above. If $(t_i,\sigma_i)$ and $(t_j,\sigma_j)$ are two pairs, then on each connected component of their overlap we have $\sigma_j=c_{ij}\sigma_i$. Since both sections are parallel,
\[
 0=\nabla\sigma_j=dc_{ij}\otimes\sigma_i.
\]
Thus $c_{ij}\in\mathbb C^*$ is constant. Comparing the two expressions for $\pi$ gives
\begin{equation}\label{tt:affine_changes}
 t_j=a_{ij}t_i+b_{ij},\qquad
 a_{ij}=c_{ij}^{-1}\in\mathbb C^*,\quad b_{ij}\in\mathbb C.
\end{equation}
The factor $a_{ij}$ therefore cancels in \eqref{tt:identification}, so the local definitions of
$k_x$ agree on overlaps. The continuity and the normalizations $k_x(x)=0$ and $D_xk_x=\id$ follow directly from the definition.
The stable plaques vary continuously in the leafwise $C^1$ topology. Since $D_xk_x=\id$, the inverse function theorem with parameters and compactness give uniform constants $r,\varepsilon>0$ with the asserted properties.
\end{proof}

\begin{lemma}\label{lemma:local_linearization}
For $p$ sufficiently close to $x$ in $W^s(x)$,
\[
 k_{fx}(fp)=S_1(x)k_x(p).
\]
\end{lemma}
\begin{proof}
Choose a connected source box $U$ and a target box $U'$ with $f(U)\subset U'$. Let $(t,\sigma)$ and $(t',\sigma')$ be the corresponding pairs. By $f^*\nabla=\nabla$ and \eqref{nt:pullback_connection}, the section $f_\#\sigma$ on $f(U)$ is parallel. Hence there is a constant $c\in\mathbb C^*$ such that
\[
 Q_1(x)\sigma(x)=c\,\sigma'(fx),\qquad x\in U.
\]
Substituting $\pi=dt\otimes\sigma$ in the source and $\pi=dt'\otimes\sigma'$ in the target into $\pi_{fx}D_xf=Q_1(x)\pi_x$, we obtain
\[
 d(t'\circ f)\otimes(\sigma'\circ f)
 =c\,dt\otimes(\sigma'\circ f).
\]
Therefore, for some constant $b\in\mathbb C$,
\begin{equation}\label{tt:affine_dynamics}
 t'\circ f=c\,t+b.
\end{equation}
In particular, $B_{fx}S_1(x)=cB_x$, and consequently
\begin{equation}\label{tt:local_linearization}
\begin{aligned}
 k_{fx}(fp)
 &=B_{fx}^{-1}\bigl(t'(fp)-t'(fx)\bigr)\\
 &=B_{fx}^{-1}c\bigl(t(p)-t(x)\bigr)
 =S_1(x)k_x(p).
\end{aligned}
\end{equation}
\end{proof}
\begin{lemma}\label{lemma affine}
The maps $k_x$ extend to a continuous family of biholomorphisms
\[
 \varphi_x\colon W^s(x)\longrightarrow E_x^s
\]
satisfying
\[
 \varphi_x(x)=0,\qquad D_x\varphi_x=\id,\qquad
 \varphi_{fx}\circ f=S_1(x)\circ\varphi_x.
\]
\end{lemma}
\begin{proof}
For $p\in W^s(x)$, choose $n$ sufficiently large that $f^np$ belongs to the local stable plaque on which $k_{f^nx}$ is defined, and set
\begin{equation}\label{tt:global_linearization}
 \varphi_x(p)=S_n(x)^{-1}k_{f^nx}(f^np).
\end{equation}
By Lemma~\ref{lemma:local_linearization}, this expression is independent of all sufficiently large choices of $n$. It defines a holomorphic map agreeing with $k_x$ near $x$.

For injectivity, if $\varphi_x(p)=\varphi_x(q)$, choose $n$ large enough that $f^np$ and $f^nq$ lie in $W^s_r(f^nx)$. The injectivity of $k_{f^nx}$ then gives $p=q$. For surjectivity, let $v\in E_x^s$ and choose $n$ so large that $\|S_n(x)v\|<\varepsilon$, where $\varepsilon$ is given by Lemma~\ref{lemma:uniform_local_coordinates}. Then
\[
 p=f^{-n}\bigl(k_{f^nx}^{-1}(S_n(x)v)\bigr)
\]
lies in $W^s(x)$ and satisfies $\varphi_x(p)=v$. Thus $\varphi_x$ is a biholomorphism. The same formula, with a single $n$ chosen locally in $(x,v)$, shows that $(x,v)\mapsto\varphi_x^{-1}(v)$ is continuous.

The continuity of these biholomorphisms and the stated normalization and equivariance follow from the iteration and the local properties in Lemma \ref{lemma:uniform_local_coordinates}. \qedhere


\end{proof}

This non-stationary linearization was essentially established by Ghys \cite[Proposition 3.1]{ghys_1995}, in the language of invariant affine structures. Our approach here is different and offers more information about the unstable holonomy maps. 
\begin{lemma}\label{lemma local holonomy}
For any $x \in X$ and $y\in W^u_{\mathrm{loc}}(x)$, let $h^u_{xy}$ be local unstable holonomy between stable plaques. Then 
\begin{equation}\label{tt:linear_holonomy}
 \varphi_y\circ h^u_{xy}\circ \varphi_x^{-1}
 =D_xh^u_{xy}
\end{equation}
on a neighborhood of $0\in E_x^s$. In particular, the conjugated local map is the restriction of a complex-linear isomorphism. 
\end{lemma}
\begin{proof}
Choose a transverse coordinate box as above containing $x$ and $y$, and shrink the stable plaques so that each holonomy segment stays in an unstable plaque of this box. Since $t$ is constant on unstable plaques,
\[
    t\circ h^u_{xy}=t,\qquad t(y)=t(x).
\]
By \eqref{tt:identification} and Lemma~\ref{lemma affine}, for $p$ sufficiently close to $x$ in $W^s(x)$,
\[
\begin{aligned}
    \varphi_y(h^u_{xy}(p))
    &=B_y^{-1}\bigl(t(h^u_{xy}(p))-t(y)\bigr)\\
    &=B_y^{-1}\bigl(t(p)-t(x)\bigr)\\
    &=B_y^{-1}B_x\varphi_x(p).
\end{aligned}
\]
Differentiating at $p=x$ and using $D_x\varphi_x=\mathrm{id}$ and $D_y\varphi_y=\mathrm{id}$ gives $D_xh^u_{xy}=B_y^{-1}B_x$.
\end{proof}
\subsection{Global holonomy and the topological conclusion}
Recall that Ghys \cite[Theorem B]{ghys_1995} proved Theorem \ref{theorem topo} under the additional assumption that $f$ is topologically transitive. One main step is 
\begin{proposition}[Proof of Theorem B in \cite{ghys_1995}, p.598]\label{prop ghys}
Let $f:X \to X$ be a holomorphic Anosov diffeomorphism of a compact manifold with $\dim_\mathbb C E^s=1$. If the unstable holonomy maps are biholomorphisms between the entire stable leaves, then $X$ is homeomorphic to $\mathbb T^{2d}$ and $f$ is topologically conjugate to a torus automorphism. 
\end{proposition}
This proposition does not use the transitivity assumption. To apply this proposition, Ghys established the global existence of the unstable holonomy by using the transitivity assumption  \cite[Lemma 5.1]{ghys_1995}. Here we prove the global existence of the unstable holonomy using a similar argument, but use Lemma \ref{lemma local holonomy} as an alternative to the  transitivity assumption. 
\begin{lemma}\label{lemma:global_unstable_holonomy}
For any $x\in X$ and $y\in W^u(x)$, the unstable holonomy is a biholomorphism
\[
    h_{xy}^u\colon W^s(x)\longrightarrow W^s(y).
\]
\end{lemma}
\begin{proof}
By composing holonomies, we may assume that $y\in W^u_{\mathrm{loc}}(x)$ is close to $x$. Let $\gamma\colon[0,1]\to W^u_{\mathrm{loc}}(x)$ be a path  from $x$ to $y$ , and let $h_t$ denote holonomy along $\gamma|_{[0,t]}$. Fix $R>0$ and set
\[
    D_R=\varphi_x^{-1}\bigl(\overline{B(0,R)}\bigr).
\]
Let $t_0$ be the supremum of times $t\in[0,1]$ such that $h_s$ is defined on all of $D_R$ for every $s\in[0,t]$.
Compactness of $D_R$ and local product structure give $t_0>0$.

For $t\in[0,1]$, define
\[
    H_t
    :=\varphi_{\gamma(t)}^{-1}\circ D_xh^u_{x\gamma(t)}\circ\varphi_x
    \colon W^s(x)\longrightarrow W^s(\gamma(t)).
\]
For $t<t_0$, Lemma~\ref{lemma local holonomy} and the identity theorem give $h_t=H_t$ on $D_R$.
Since $(t,p)\mapsto H_t(p)$ is continuous on $[0,1]\times D_R$, we have
\[
    h_t\longrightarrow H_{t_0}
    \quad\text{uniformly on }D_R
    \quad\text{as }t\uparrow t_0.
\]

For each $p\in D_R$, the path $t\mapsto h_t(p)$ lies in a single unstable plaque for $t<t_0$ sufficiently close to $t_0$, so its limit belongs to the same plaque. Thus the holonomy extends to $t_0$. Compactness of $H_{t_0}(D_R)$ and local product structure then allow further continuation if $t_0<1$, contradicting the definition of $t_0$. Hence $t_0=1$, and $h_1$ is defined on all of $D_R$. Since $R$ was arbitrary, $h_1$ is globally defined and
coincides with the biholomorphism $H_1$.
\end{proof}
Proposition~\ref{prop ghys} and Lemma~\ref{lemma:global_unstable_holonomy} now prove Theorem~\ref{theorem topo}.



\section{Proof of Theorem~\ref{theorem anosov torus}}\label{sec torus Anosov}
The proof rests on three observations. First, all deck
transformations of the universal cover $\wt X$ are biholomorphic. Second, each deck translation can be decomposed into stable and unstable components. When the stable and unstable foliations are holomorphic, both components act biholomorphically, and their translation vectors form dense subgroups of the corresponding linear subspaces. Third, locally uniform limits of holomorphic maps are holomorphic. Taking limits therefore shows that all translations act biholomorphically, which forces the corresponding complex structure to be translation-invariant.

Let $X$ and $f$ satisfy the assumptions of
Theorem~\ref{theorem anosov torus}. By the Franks--Manning theorem
\cite{franks_1969,manning_1974}, there exist a real torus
\[
  N=V/\Gamma,
  \qquad V\simeq\R^{2d},\quad \Gamma\simeq\Z^{2d},
\]
a hyperbolic affine automorphism $\alpha$ of $N$, and a homeomorphism
$h\colon X\to N$ such that
\[
  h\circ f=\alpha\circ h.
\]
The affine map $\alpha$ has a fixed point. After conjugating it by a translation, we may therefore assume that it is induced by a linear automorphism $A\colon V\to V$ satisfying $A(\Gamma)=\Gamma$.

Let $\wt X$ be the universal cover of $X$, and let $F,A,H$ denote compatible
lifts of $f,\alpha,h$, respectively. Then
\begin{equation}\label{equation Anosov equivariance}
  H\circ F=A\circ H.
\end{equation}
The homeomorphism $H$ maps lifted stable and unstable leaves to affine
subspaces parallel to the stable and unstable subspaces of $A$, respectively.

Write $V=V^s\oplus V^u$ for the hyperbolic splitting of $A$, and let
$p_s\colon V\to V^s$ and $p_u\colon V\to V^u$ be the corresponding
projections.

\begin{lemma}\label{lemma projection density}
We have
\[
  \ol{p_s(\Gamma)}=V^s,
  \qquad
  \ol{p_u(\Gamma)}=V^u.
\]
\end{lemma}

\begin{proof}
We give the argument for $p_s$. Let $\pi\colon V\to V/\Gamma$ be the quotient map. Since the unstable
foliation of the hyperbolic toral automorphism $A$ is minimal, $\pi(V^u)$ is
dense in $V/\Gamma$. Equivalently, $V^u+\Gamma$ is dense in $V$. Passing to
the quotient $V/V^u\simeq V^s$ shows that $p_s(\Gamma)$ is dense in $V^s$.
Similarly, the minimality of the stable foliation implies that $p_u(\Gamma)$
is dense in $V^u$.
\end{proof}
For $v\in V$, define
\[
  \Phi_v:=H^{-1}\circ L_v\circ H\colon\wt X\to\wt X,
\]
where $L_v$ denotes translation by $v$.

\begin{lemma}\label{lemma anosov translation}
The map $\Phi_v$ is biholomorphic for every $v\in V$.
\end{lemma}

\begin{proof}
For every $\gamma\in\Gamma$, the map $\Phi_\gamma$ is a deck transformation
and hence is biholomorphic. We first show that $\Phi_{p_s\gamma}$ is
biholomorphic.

Since $E^s$ and $E^u$ are holomorphic subbundles, the lifted stable and
unstable foliations are transverse holomorphic foliations. Fix $o\in\wt X$
with $H(o)=0$ and set
\[
  S=\widetilde W^s(o),
  \qquad
  U=\widetilde W^u(o).
\]

The conjugacy gives global product structure on $\widetilde X$. Define
\[
 r_s(x)=\widetilde W^u(x)\cap S,\qquad
 r_u(x)=\widetilde W^s(x)\cap U.
\]
Both projections are holomorphic. The pair $(r_s,r_u)$ is a biholomorphism with inverse
\[
 B\colon S\times U\longrightarrow\widetilde X,
 \qquad B(s,u)=\widetilde W^u(s)\cap\widetilde W^s(u).
\]
Since $\Phi_\gamma$ preserves both lifted foliations, it has the following form in the product coordinates given by $B$:
\[      B^{-1}\circ\Phi_\gamma\circ B(s,u)
  =\bigl(\rho^s_\gamma(s),\rho^u_\gamma(u)\bigr),\]
where $\rho^s_\gamma\colon S\to S$ and $\rho^u_\gamma\colon U\to U$ are biholomorphisms.

Moreover, since $H$ preserves both lifted foliations, we have
\[ H(B(s,u))=H(s)+H(u),
  \qquad H(s)\in V^s,\quad H(u)\in V^u,\]
and 
\begin{align*}
B^{-1}\circ\Phi_\gamma\circ B(s,u)
  &=B^{-1}\circ H^{-1}\circ L_\gamma\circ H\circ B(s,u)\\
  &=\bigl(H^{-1}(H(s)+p_s\gamma),
      H^{-1}(H(u)+p_u\gamma)\bigr).
\end{align*}
It follows that
\[
  \rho^s_\gamma(s)=H^{-1}(H(s)+p_s\gamma)
\]
is holomorphic. Consequently,
\[B^{-1}\circ\Phi_{p_s\gamma}\circ B(s,u)
  =B^{-1}\circ H^{-1}\circ L_{p_s\gamma}\circ H\circ B(s,u)=(\rho^s_\gamma(s),u ).\]
The right-hand side is holomorphic, and hence $\Phi_{p_s\gamma}$ is holomorphic.

Now let $a_s\in V^s$. By Lemma~\ref{lemma projection density}, we can choose
$\gamma_j\in\Gamma$ such that $p_s\gamma_j\to a_s$. Then $\Phi_{p_s\gamma_j}\to\Phi_{a_s}$ locally uniformly, so $\Phi_{a_s}$ is holomorphic. Similarly, the inverses $\Phi_{-p_s\gamma_j}$ converge locally uniformly to $\Phi_{-a_s}=\Phi_{a_s}^{-1}$, which is therefore holomorphic as well. Thus $\Phi_{a_s}$ is biholomorphic.

The same argument in the unstable direction shows that $\Phi_{a_u}$ is
biholomorphic for every $a_u\in V^u$. Since every $v\in V$ has a unique
decomposition $v=a_s+a_u$ and
$\Phi_v=\Phi_{a_s}\circ\Phi_{a_u}$, the conclusion follows.
\end{proof}

The maps $\{\Phi_v:v\in V\}$ define a continuous, free, and transitive action
of $V$ on $\wt X$ by biholomorphisms. To show that this action is smooth, we use the following consequence of the Bochner--Montgomery theorem \cite{bochner_montgomery_1946}.

\begin{theorem}[Bochner--Montgomery]\label{theorem BM}
Let $G$ be a locally compact topological group acting continuously and effectively
on a connected smooth manifold $M$ by $C^\infty$ diffeomorphisms. Then $G$ is a Lie group and the action map $G\times M\to M$ is smooth.
\end{theorem}
In our case $G=(V,+)$ is a finite-dimensional Lie group acting freely on $\widetilde X$. Theorem~\ref{theorem BM} therefore makes the action smooth.  Choose $x_0\in\wt X$ with
$H(x_0)=0$ and consider the orbit map
\[
  \Theta\colon V\to\wt X,
  \qquad
  v\longmapsto\Phi_v(x_0)=H^{-1}(v).
\]
The kernel of $D_0\Theta$ is the Lie algebra of the stabilizer of $x_0$, which is zero because the action is free. Since $\dim_{\mathbb R}V=\dim_{\mathbb R}\widetilde X$, the orbit map is a local diffeomorphism. It is bijective by freeness and transitivity, and hence is a global $C^\infty$ diffeomorphism. Consequently, $H=\Theta^{-1}$ is a $C^\infty$ diffeomorphism.

Let $J_0:=H_*J_{\widetilde X}$ be the induced complex structure on $V$. Since all
translations are $J_0$-holomorphic, $J_0$ is translation-invariant. The conjugacy relation
\[
  H\circ F\circ H^{-1}=A
\]
and the holomorphicity of $F$ imply
\[
  AJ_0=J_0A.
\]
Thus $A$ is complex linear with respect to $J_0$, $\Gamma$ acts by holomorphic
translations, and
\[
  h\colon X\longrightarrow (V,J_0)/\Gamma
\]
is a biholomorphic conjugacy. This proves
Theorem~\ref{theorem anosov torus}.

\section{Proof of Theorem~\ref{theorem distribution}}\label{sec:distribution}

We begin with a criterion for the holomorphicity of sections. Recall that a sequence $f_n\in L^1_{\mathrm{loc}}(\Omega)$, where
$\Omega\subset\mathbb C^d$ is open, converges to $f\in L^1_{\mathrm{loc}}(\Omega)$ in the sense of distributions if
\[
    \int_\Omega f_n\varphi\,dV\longrightarrow
    \int_\Omega f\varphi\,dV
    \qquad\text{for every }\varphi\in C_c^\infty(\Omega),
\]
where $dV$ denotes Lebesgue measure. For sections of a holomorphic vector bundle, distributional convergence is understood coefficientwise in local
holomorphic frames.

\begin{lemma}\label{lem:holomorphic_limit_l2}
Let $X$ be a complex manifold, and let $E\to X$ be a holomorphic vector
bundle. Suppose that $s\in C^0(X,E)$ and $s_n\in C^0(X,E)$ satisfy
\[
    s_n\longrightarrow s
    \quad\text{and}\quad
    \bar\partial_Es_n\longrightarrow0
    \qquad\text{in the sense of distributions}.
\]
Then $s$ is a holomorphic section of $E$.
\end{lemma}

\begin{proof}
Since differentiation is continuous in the distribution topology, we have
$\bar\partial_Es_n\to\bar\partial_Es$ in the sense of distributions.
Uniqueness of the distributional limit gives $\bar\partial_Es=0$.
In a local holomorphic frame of $E$, the coefficients of $s$ are therefore
holomorphic by the distributional form of the Dolbeault--Grothendieck lemma
\cite[Chapter~I, Lemma~(3.29)(a), with $p=q=0$]{demailly_2012}. Thus $s$ is holomorphic.
\end{proof}

Let $f\colon X\to X$ satisfy the assumptions of Theorem~\ref{theorem distribution}, so $\dim_{\mathbb C}E^s=1$. Put $d=\dim_{\mathbb C}X$. If $d=2$, both invariant bundles are holomorphic by \cite[Proposition~2.2]{ghys_1995}. We may therefore assume $d\ge3$.  By Theorem~\ref{theorem topo}, there exist a real torus
\[
  N=V/\Gamma,
  \qquad V\simeq\R^{2d},\quad \Gamma\simeq\Z^{2d},
\]
a hyperbolic linear automorphism $\alpha$ of $N$ and a H\"older homeomorphism $h\colon X\to N$ such that
\[
  h\circ f=\alpha\circ h.
\]
Let $\wt X$ be the universal cover of $X$, and let $F,A,H$ denote compatible lifts of $f,\alpha,h$, respectively. Then
\begin{equation}
  H\circ F=A\circ H.
\end{equation}
\begin{lemma}\label{lem:ha_closed_form}
There exist a nowhere-vanishing closed holomorphic 1-form $\theta$ on $X$ and a constant $\mu\in\mathbb C$ with $0<|\mu|<1$ such that
\begin{equation}\label{eq:ha_form_invariance}
    \ker\theta=E^u,
    \qquad f^*\theta=\mu\theta.
\end{equation}
\end{lemma}

\begin{proof}
Write $V=V^s\oplus V^u$ for the invariant splitting of $A$ and set $p=H^{-1}(0)$, $S=\widetilde W^s(p)$. Then $F(p)=p$ and $F(S)=S$. Identify $E_p^s$ with $\mathbb C$ and write $\varphi=\varphi_p\colon S\to\mathbb C$. By Lemma~\ref{lemma affine}, 
\[\varphi\circ F|_S\circ\varphi^{-1}(z)=\mu z    \qquad\text{for some }0<|\mu|<1.\]
Define
\[
    q\colon\widetilde X\to\mathbb C,
    \qquad
    q(x)=\varphi\bigl(\widetilde{ W}^u(x)\cap S\bigr).
\]
The topological conjugacy gives global product structure, so the intersection defining $q$ exists and is unique. Locally, this projection onto $S$ is holonomy of the holomorphic foliation $\mathcal W^u$, followed by $\varphi$. Hence $q$ is a holomorphic submersion satisfying
\begin{equation}\label{eq:ha_universal_form}
    \ker dq=\widetilde E^u,
    \qquad q\circ F=\mu q.
\end{equation}
For each $\gamma\in\Gamma$, let
\[\Phi_\gamma=H^{-1}\circ L_\gamma\circ H\]
be the corresponding deck transformation, where $L_\gamma$ denotes translation by $\gamma$ on $V$. Then $\Phi_\gamma\colon\widetilde X\to\widetilde X$ is a biholomorphism preserving the lifted stable and unstable foliations. In particular, $\Phi_\gamma$ induces a biholomorphism $g_\gamma$ of $\mathbb C$ such that
\[
    q\circ \Phi_\gamma=g_\gamma\circ q.
\]
For $\gamma\ne0$, the map $\Phi_\gamma$ cannot preserve a lifted unstable leaf. Indeed, that would give $\gamma\in V^u\cap\Gamma$; but then $A^{-n}\gamma\in\Gamma$ would converge to zero, contradicting discreteness of $\Gamma$. Thus $g_\gamma$ has no fixed point and is therefore a translation map of $\mathbb C$. It follows that $dq$ is invariant under the deck group and descends to a nowhere-vanishing closed holomorphic 1-form $\theta$ on $X$. The identities in \eqref{eq:ha_form_invariance} follow from \eqref{eq:ha_universal_form}.
\end{proof}

Let $e$ be the continuous vector field uniquely determined by $e(x)\in E^s(x)$ and $\theta_x(e(x))=1$. We show that $e$ is holomorphic by constructing a sequence of smooth vector fields converging uniformly to $e$ whose antiholomorphic derivatives converge to zero in $L^2$, and then applying Lemma~\ref{lem:holomorphic_limit_l2}.
\begin{lemma}\label{lemma L2}
The vector field $e$ is holomorphic. In particular, $E^s$ is a holomorphically trivial line subbundle of $T^{1,0}X$.
\end{lemma}

\begin{proof}
The normalization of $e$ and the invariance of $E^s$ give
\[  D_xf^n e(x)=\mu^n e(f^n x),
    \qquad n\ge0.\]
Since $\theta$ is nowhere vanishing, we may choose a smooth section $e_0$ of $T^{1,0}X$ such that
$\theta(e_0)=1$. The section $e_0$ need not take values in $E^s$. Set
\begin{equation}\label{eq:ha_smooth_approximants}
    e_n(x)=\mu^n(D_xf^n)^{-1}e_0(f^n x),
    \qquad
    U_n(x)=D_xf^n|_{E^u_x}.
\end{equation}
Since $e_0-e$ takes values in $E^u$, we have
\[    e_n(x)-e(x)
    =\mu^n U_n(x)^{-1}(e_0-e)(f^n x).\]
Uniform expansion of $E^u$ therefore implies that $e_n\to e$ uniformly.

Since $f$ is holomorphic, differentiating \eqref{eq:ha_smooth_approximants} gives
\begin{equation}\label{eq deri}
    (\bar\partial e_n)_x=\mu^n(D_xf^n)^{-1}(\bar\partial e_0)_{f^nx}
           \circ\overline{D_xf^n}.
\end{equation}
Moreover, since $\theta$ is holomorphic and $\theta(e_0)=1$, we have
\[
    0=\bar\partial\bigl(\theta(e_0)\bigr)
     =\theta(\bar\partial e_0),
\]
which means that $\bar\partial e_0$ takes values in $E^u=\ker\theta$. Thus \eqref{eq deri} reduces to 
\begin{equation}\label{eq:ha_dbar_chain_rule}
    (\bar\partial e_n)_x
    =\mu^n U_n(x)^{-1}
      (\bar\partial e_0)_{f^nx}\circ\overline{D_xf^n}.
\end{equation}

Choose a smooth Hermitian metric on the holomorphic bundle $E^u=\ker\theta$ and extend it to $T^{1,0}X$ by requiring
\[ E^u\perp\mathbb C e_0,
    \qquad |e_0|=1.\]
All norms below are taken with respect to this metric. The letter $C$ denotes a positive constant independent of $x$ and $n$, whose value may change from one occurrence to the next.

Relative to the smooth orthogonal splitting $E^u\oplus\mathbb C e_0$, the relation $(f^n)^*\theta=\mu^n\theta$ gives the block matrix
\[
    D_xf^n=
    \begin{pmatrix}
        U_n(x)&b_n(x)\\
        0&\mu^n
    \end{pmatrix}.
\]

Let $\nu$ denote the associated real volume form. The real Jacobian is
\begin{equation}\label{eq:ha_real_jacobian}
    J_\nu(f^n)(x)=|\mu|^{2n}|\det U_n(x)|^2,
\end{equation}
where $|\det U_n(x)|$ denotes the modulus of the complex determinant computed in unitary frames.

Let $m(U_n(x))=\|U_n(x)^{-1}\|^{-1}$ denote the conorm of $U_n(x)$. Uniform hyperbolicity gives $\lambda>1$ such that, for every $x\in X$ and $n\geq1$,
\begin{equation}\label{eq U}
    m(U_n(x))\geq C^{-1}\lambda^n,
    \qquad
    \|D_xf^n\|\leq C\|U_n(x)\|.
\end{equation}

Choose $n_0$ so that $m(U_n(x))\geq1$ for all $x\in X$ and $n\geq n_0$. For such $n$, all singular values of $U_n(x)$ are at least one, and hence
\begin{equation}\label{eq U2}
      |\det U_n(x)|\geq\|U_n(x)\|.
\end{equation}

Since $e_0$ is smooth and $X$ is compact, $\bar\partial e_0$ is uniformly bounded. Combining \eqref{eq:ha_dbar_chain_rule}, \eqref{eq:ha_real_jacobian}, \eqref{eq U}, and \eqref{eq U2}, we obtain, for $n\geq n_0$,
\begin{equation}\label{eq:ha_jacobian_bound}
    \begin{aligned}
    |\bar\partial e_n(x)|^2
    &\le C|\mu|^{2n}\|U_n(x)^{-1}\|^2
       |\bar\partial e_0(f^nx)|^2\|D_xf^n\|^2\\
    &\le C\frac{|\mu|^{2n}\|U_n(x)\|^2}{m(U_n(x))^2}
    \le C\lambda^{-2n}J_\nu(f^n)(x).
    \end{aligned}
\end{equation}
Integrating \eqref{eq:ha_jacobian_bound} yields
\begin{equation}\label{eq:ha_l2_convergence}
    \begin{aligned}
        \|\bar\partial e_n\|_{L^2(\nu)}^2
        &\leq C\lambda^{-2n}\int_X J_\nu(f^n)\,d\nu\\
        &=C\lambda^{-2n}\nu(X)\longrightarrow0.
    \end{aligned}
\end{equation}
Here we use only that $f^n$ is an orientation-preserving diffeomorphism;
the volume form $\nu$ need not be invariant.

The $L^2$ convergence of $\bar\partial e_n$ to zero implies convergence in the sense of distributions. Likewise, the uniform convergence $e_n\to e$ implies distributional convergence. Lemma~\ref{lem:holomorphic_limit_l2} therefore makes $e$ holomorphic. Since $e$ is nowhere vanishing and spans $E^s$, this bundle is holomorphically trivial. This proves Theorem~\ref{theorem distribution}.
\end{proof}
\section{An example with non-holomorphic $E^s$ and $E^u$}\label{sec example}
We conclude with a holomorphic Anosov diffeomorphism $f\colon M\to M$ on a compact nilmanifold of complex dimension six whose stable and unstable distributions are both smooth but non-holomorphic. The complex structure is left-invariant, and $f$ is induced by an automorphism of the underlying real nilpotent Lie group. So this example fits into Ghys's conjecture.

The construction is a complex analogue of the Borel--Smale Anosov automorphism on a nilmanifold of real dimension six \cite{smale_1967}.

Equip $G=\C^6$ with its standard complex structure and the multiplication
\[
\begin{aligned}
 &(x,y,z,a,b,c)\cdot(x',y',z',a',b',c')\\[-2pt]
 &\qquad=(x+x',\,y+y',\,z+z'+\bar x y',\,
          a+a',\,b+b',\,c+c'+\bar a b').
\end{aligned}
\]
This is a simply connected real two-step nilpotent Lie group.
Every left translation is holomorphic. Thus the standard complex structure is left-invariant. 

Let $\mathcal R=\Z[i,\sqrt3]$, and let $\sigma:\mathcal R\to\mathcal R$
be the automorphism fixing $i$ and sending $\sqrt3$ to $-\sqrt3$. Set
\[
\Gamma=\bigl\{(\alpha,\beta,\gamma,\sigma\alpha,
  \sigma\beta,\sigma\gamma):
  \alpha,\beta,\gamma\in\mathcal R\bigr\}.
\]
Since $\sigma$ commutes with complex conjugation and the embedding
$\alpha\mapsto(\alpha,\sigma\alpha)$ identifies $\mathcal R$ with
a lattice in $\C^2$, it is straightforward to check that $\Gamma$ is a lattice of $G$.
Thus $M=\Gamma\backslash G$ is a compact
complex manifold.

Put $\lambda=2+\sqrt3$ and define
\[
F(x,y,z,a,b,c)=
(\lambda^2x,\lambda^{-1}y,\lambda z,
 \lambda^{-2}a,\lambda b,\lambda^{-1}c).
\]
The weights respect both cross terms, so $F:G \to G$ is a holomorphic group
automorphism. Since $\lambda^{\pm1}\in\mathcal R$ and $\sigma(\lambda)=\lambda^{-1}$, it preserves $\Gamma$ and induces
a biholomorphism $f:M\to M$.

The left-invariant $(1,0)$ frame, which descends to $M$, is
\[
\begin{aligned}
X&=\partial_x,&Y&=\partial_y+\bar x\,\partial_z,&Z&=\partial_z,\\
A&=\partial_a,&B&=\partial_b+\bar a\,\partial_c,&C&=\partial_c.
\end{aligned}
\]
The multipliers of $DF$ in this frame are $\lambda^2,\lambda^{-1},\lambda,\lambda^{-2},\lambda,\lambda^{-1}$, so $f$ is Anosov. The stable and unstable distributions are given by
\[
E^s=\operatorname{span}_{\C}\{Y,A,C\},
\qquad
E^u=\operatorname{span}_{\C}\{X,Z,B\}.
\]
Both invariant bundles are smooth but not holomorphic. Indeed,
\[
\bar\partial Y=d\bar x\otimes Z,
\qquad
\bar\partial B=d\bar a\otimes C,
\]
and hence
\[
(\bar\partial Y)(\partial_{\bar x})=Z\notin E^s,
\qquad
(\bar\partial B)(\partial_{\bar a})=C\notin E^u.
\]
For a holomorphic subbundle, the $\bar\partial$-derivative of every
smooth section must take values in that subbundle.
Thus neither $E^s$ nor $E^u$ is holomorphic.
\bibliographystyle{amsalpha}
\bibliography{ref}

\end{document}